\documentclass[12pt]{amsart}

\usepackage[utf8]{inputenc}
\usepackage{amsmath}
\usepackage{amsthm}
\usepackage{amsfonts}
\usepackage[margin=1in]{geometry}
\usepackage[colorlinks=true]{hyperref}
\usepackage{tikz}
\usepackage[parfill]{parskip}
\usepackage[all]{xy}
\usepackage{titlesec} 
\usepackage{setspace}

\newcommand{\Ext}{\mathrm{Ext}}

\newcommand{\Hom}{\mathrm{Hom}}
\newcommand{\End}{\mathrm{End}}

\newcommand{\Tor}{\mathrm{Tor}}

\newcommand{\ann}{\mathrm{ann}}
\newcommand{\m}{\mathfrak{m}}

\newcommand{\module}{\mathrm{mod}}

\newcommand{\MCM}{\mathrm{MCM}}

\newcommand{\CM}{\mathrm{CM}}

\newcommand{\add}{\mathrm{add}}

\newcommand{\proj}{\mathrm{proj}}

\newcommand{\depth}{\mathrm{depth }}
\newcommand{\gldim}{\mathrm{gldim }}
\newcommand{\pd}{\mathrm{pd }}
\newcommand{\CMdd}{\mathrm{CMdomdim }}

\newcommand{\ZZ}{\mathbb{Z}}

\newcommand{\imof}{\mathrm{im}}

\newtheorem{theorem}{Theorem}[section]
\newtheorem*{theorem*}{Theorem}
\newtheorem{lemma}[theorem]{Lemma}
\newtheorem*{lemma*}{Lemma}
\newtheorem{proposition}[theorem]{Proposition}
\newtheorem*{proposition*}{Proposition}
\newtheorem{corollary}[theorem]{Corollary}
\newtheorem*{corollary*}{Corollary}

\theoremstyle{definition}
\newtheorem{definition}[theorem]{Definition}

\newtheorem{remark}[theorem]{Remark}

\newtheorem{chunk}[theorem]{}

\author{\"{O}zg\"{u}r Esentepe}

\address{Department of Mathematics, University of Connecticut, Storrs, CT, 06269}
\email{ozgur.esentepe@uconn.edu}

\titleformat{\chapter}[block]{\Huge\scshape\bf\centering}{\thechapter.}{1em}{} 
\titleformat{\section}[block]{\large\scshape\bf\centering}{\thesection.}{1em}{} 
\titleformat{\subsection}[block]{\scshape\bf\centering}{\thesubsection.}{1em}{} 

\title{\scshape A Note On the Global Dimension of Shifted Orders}

\begin{document}
\maketitle
\begin{abstract}
    We consider the dominant dimension of an order over a Cohen-Macaulay ring in the category of centrally Cohen-Macaulay modules. There is a canonical tilting module in the case of positive dominant dimension and we give an upper bound on the global dimension of its endomorphism ring.
\end{abstract}

\section{Introduction}

In a recent article ``\emph{Special tilting modules for algebras with positive dominant dimension}" \cite{PresslandSauter}, Matthew Pressland and Julia Sauter study what they advertise in their title. They call such tilting modules \emph{shifted modules} and they study their \emph{shifted algebras}. One particular theorem they prove is the following.

\begin{theorem*}[\cite{PresslandSauter}, Proposition 2.13]
Let $\Gamma$ be a finite dimensional algebra with dominant dimension $d$ and let $0 \leq k \leq d$. Then, for the $k$-shifted algebra $B_k$ of $\Gamma$, we have
\begin{align*}
    \gldim B_k \leq \gldim \Gamma.
\end{align*}
\end{theorem*}

The aim of this note is to show that this theorem holds true in the setting of Cohen-Macaulay representation theory. In our setting, we replace finite dimensional algebras over a field with orders over a Cohen-Macaulay local ring with canonical module. Therefore, we replace the module category of an algebra with the category of Cohen-Macaulay modules over an order. We study the dominant dimension in this category. This idea appears also in \cite{AuslanderRoggenkamp} where the authors characterize orders of finite lattice type.

We spend Section 2 to introduce the necessary background and notation in Cohen-Macaulay representation theory while we emphasize the similarities between the theory of finite dimensional algebras and the theory of orders over Cohen-Macaulay local rings.

In Section 3, we introduce the notion of Cohen-Macaulay dominant dimension. This is a notion which proves useful only in the noncommutative setting as the dominant dimension of a commutative Gorenstein local ring is infinite while it is zero when we have a Cohen-Macaulay local ring which is not Gorenstein. We show computations for one noncommutative example and discuss the Cohen-Macaulay Nakayama conjecture. 

When an order over a Cohen-Macaulay local ring has positive Cohen-Macaulay dominany dimension, there is a special family of tilting modules. We call their endomorphism rings \textit{shifted orders}. While in general derived equivalence preserves finiteness of global dimension, the exact value of the global dimension is not preserved. We prove that the global dimension of the shifted orders can not be greater than the original order.

\begin{theorem*}
Suppose that $\Lambda$ is an $n$-canonical QF-3 order of finite global dimension over a Cohen-Macaulay local ring $R$ with canonical module. Then, the $\ell^{\rm{th}}$-shifted order $\Gamma_\ell$ has global dimension
\begin{align*}
\gldim \Gamma_\ell \leq \gldim \Lambda.
\end{align*}
\end{theorem*}

\textbf{Acknowledgements.} I would like to thank Graham Leuschke for adopting me as a PhD student. When I visited him in 2018 in Syracuse, he mentioned an ICRA2016 talk by William Crawley-Boevey. Main ideas for this paper appeared when I was trying to understand Crawley-Boevey's talk. But then I forgot about the original problem and this paper was born. I also thank Osamu Iyama, Benjamin Briggs, Vincent G{\'e}linas and Louis-Philippe Thibault for many fruitful discussions. Finally, many thanks to Delmie Esentepe for delaying the final version of this paper for more than a year by being born during a pandemic and allowing me to rewrite it for better exposition.

\section{Background and Motivation}

The aim is that this short note is as self-contained as possible. We will start with the definitions from commutative world and then discuss how it generalizes to the noncommutative landscape.

\begin{chunk}\label{mainsetting}\textbf{Setting, conventions and notation.}
Throughout this paper, we denote by $R$ a commutative Noetherian local ring $(R, \m)$ of Krull dimension $d$ and by $\Lambda$ a module-finite $R$-algebra with Jacobson radical $J$. We always consider right $\Lambda$-modules and we treat $\Lambda^{op}$-modules as left $\Lambda$-modules. Unless otherwise stated all modules are finitely generated.
\end{chunk}

\begin{chunk}\textbf{Commutative algebra.} Let $M$ be an $R$-module. We say that a ring element $r \in R$ is a non-zerodivisor on $M$ if $rm \neq 0$ for any nonzero $m \in M$. A regular sequence on $M$ is a sequence $r_1, \ldots, r_\ell$ of elements in $\m$ such that $r_1$ is a non-zerodivisor on $M$ and $r_j$ is a non-zerodivisor on $M/(r_1, \ldots, r_{j-1})M$ for $1 < j \leq \ell$. The maximal length of a regular sequence on $M$ is an important invariant and it is called the depth of $M$. There is also a homological characterisation of depth:
\begin{align*}
  \depth_R M  = \inf\left\{ i \colon \Ext_R^i(R/\m,M) \neq 0 \right\}.
 \end{align*}
Another important invariant of $M$ is the dimension of $M$ which is defined as the Krull dimension of the quotient ring $R/\ann_RM$. These two invariants satisfy the following inequalities for any nonzero $M$: $    \depth_R M \leq \dim _RM \leq d.
$ We say that $M$ is a \textit{maximal Cohen-Macaulay} module if there is equality. We also assume that the zero module is maximal Cohen-Macaulay and denote the category of maximal Cohen-Macaulay modules by $\MCM(R)$.

If $M$ has finite projective dimension, then the depth and the projective dimension of $M$ add up to $d$. This is the \textit{Auslander-Buchsbaum formula.} As an immediate consequence, we see that projective modules are maximal Cohen-Macaulay and non-projective maximal Cohen-Macaulay modules have infinite projective dimension. Therefore, if $R$ has finite global dimension (for instance when it is the coordinate ring of a smooth irreducible variety) we see that every maximal Cohen-Macaulay module is projective. The converse is also true.

We say that $R$ is a \emph{Cohen-Macaulay} ring if the regular module $R$ is maximal Cohen-Macaulay. For a Cohen-Macaulay ring $R$, we denote by $\omega_R$ an $R$-module which has finite injective dimension, is maximal Cohen-Macaulay and satisfies $\dim \Ext_R^d(k, \omega_R) = 1$ where $k = R/\m$ is the residue field. This module, if exists, is unique up to isomorphism and we call it the \textit{canonical module} of $R$. In this paper, we will always be interested in Cohen-Macaulay base rings with canonical module. Existence of a canonical module is not a big ask. For example, quotients of power series rings have canonical modules.

The following conditions are equivalent for a Cohen-Macaulay local ring $R$ with canonical module $\omega_R$. (1) $\omega_R$ is free of rank one, (2) $R$ has finite injective dimension as a module over itself, (3) $R$ has injective dimension $d$ and finally (4) $\Ext_R^n(k,R) = 0$ for $n \neq 0, d$ and $\Ext_R^d(k,R) \cong k$. In this case we say that $R$ is a \textit{Gorenstein} ring.

Over a Gorenstein ring, a module $M$ is maximal Cohen-Macaulay if and only if $\Ext_R^i(M,R) = 0$ for all $i>0$.
\end{chunk}

\begin{chunk}
\textbf{Noncommutative algebra.} In the noncommutative landscape there are different notions of Gorenstein rings and maximal Cohen-Macaulay modules. We start with Iwanaga-Gorenstein rings. Let $\Lambda$ be as in Setting \ref{mainsetting}.

\begin{definition}
    We say that $\Lambda$ is an \textit{Iwanaga-Gorenstein} ring if it has finite injective dimension as a module over itself and as a module over its opposite ring. In this case, a finitely generated $\Lambda$-module $M$ is called \textit{maximal Cohen-Macaulay} if $\Ext_\Lambda^i(M,\Lambda) = 0$ for all $i > 0 $.
\end{definition}
\begin{remark}
In different contexts, these modules are also called Gorenstein-projective, Gorenstein dimension zero or totally reflexive. This definition of maximal Cohen-Macaulay modules follow Buchweitz's soon-to-be published famous manuscript \cite{B} where the triangle equivalence between the stable category of maximal Cohen-Macaulay modules and the singularity category was proved. It also follows Iyama's 2018 ICM plenary talk \cite{Iyama} in Rio on tilting Cohen-Macaulay representations. However, this is not the definition we will follow in this paper. 
\end{remark} 
\begin{remark}\label{every-module-over-selfinjective}
If $R$ is a field and $\Lambda$ is a self-injective $R$-algebra, then we see that every module is maximal Cohen-Macaulay.
\end{remark}
We will now give another approach to maximal Cohen-Macaulay modules over noncommutative rings. Again, assume that $\Lambda$ is as in Setting \ref{mainsetting} and let $M$ be a $\Lambda$-module. A theorem from 2002 due to Shiro Goto and Kenji Nishida \cite[Corollary 3.2]{GotoNishida} says that the depth of $M$ as an $R$-module does not depend on the base ring. That is, if $S$ is another commutative Noetherian ring over which $\Lambda$ is a module-finite algebra, then $\depth_R M = \depth_S M$. Indeed, one has the following equality:
\begin{align*}
\depth_R M =    \inf \left\{ i \colon \Ext_\Lambda^i(\Lambda/J, M) \neq 0 \right\}
\end{align*}
where the right hand side does not depend on $R$. Thus, when we talk about the depth of a $\Lambda$-module, we can drop the subscript $R$. According to \cite[Section 2]{IyamaReiten}, the dimension of $M$ is also independent of the base ring. Therefore, it makes sense to consider the following definition.
\begin{definition}
A $\Lambda$-module is said to be \textit{maximal Cohen-Macaulay} if it is maximal Cohen-Macaulay as an $R$-module. We say that $\Lambda$ is an \textit{$R$-order} if it is maximal Cohen-Macaulay as a module over itself.
\end{definition}

\begin{remark}\label{every-module}
According to this definition, if $R$ is a field, orders are precisely finite dimensional algebras. And the category of maximal Cohen-Macaulay modules over an order coincides with the module category: every module is maximal Cohen-Macaulay.
\end{remark}

This is the definition of maximal Cohen-Macaulay modules we will consider in this paper. We will denote the category of maximal Cohen-Macaulay $\Lambda$-modules by $\CM(\Lambda)$.
\end{chunk}
\begin{chunk}
\textbf{When do the two definitions agree?} Remarks \ref{every-module-over-selfinjective} and \ref{every-module} tell us that if the base ring is a field, the two definitions of maximal Cohen-Macaulay modules agree when the algebra is self-injective. When the base ring is a Cohen-Macaulay local ring with canonical module $\omega_R$, we have a similar phenomenon.
\begin{definition}
We call the $\Lambda$-bimodule $\omega=\Hom_R(\Lambda, \omega_R)$ the \textit{canonical (bi)module} of $\Lambda$. We say that $\Lambda$ is a \textit{Gorenstein order} if $\omega$ is a projective $\Lambda$-module.
\end{definition}
Note that if $\Lambda = R$, then $R$ is Gorenstein if and only if $R$ is a Gorenstein order. So, this is another generalization of Gorenstein rings to the noncommutative world.

If $\Lambda$ is a Gorenstein $R$-order, then the two definitions of maximal Cohen-Macaulay modules agree. Moreover, for a Gorenstein order, the Auslander-Buchsbaum formula holds: for every $\Lambda$-module $M$ of finite projective dimension, the depth of $M$ and the projective dimension of $M$ adds up to the Krull dimension of $R$.

In general, when $\Lambda$ is an $R$-order, the canonical bimodule plays an important role in $\CM (\Lambda)$: for every $X \in \CM(\Lambda)$, we have $\Ext_\Lambda^{>0}(X, \omega)=0$. This means that $\omega$ is an injective object in $\CM(\Lambda)$. Moreover, $\omega$ is in fact an additive generator for the subcategory of injective objects in $\CM(\Lambda)$.

The functor $\Hom_\Lambda(-,\omega)$ is isomorphic to $\Hom_R(-,\omega_R)$ on $\CM(\Lambda)$ by hom-tensor adjunction and it is an exact duality from $\CM(\Lambda)$ to $\CM(\Lambda^{\rm op})$. Given a $\Lambda$-module $M$, consider the $\Lambda^{\rm op}$-module $\Hom_\Lambda(M, \omega)$ and take a projective resolution of it as a $\Lambda^{\rm op}$. Then, dualizing it again with the functor $\Hom_R(-,\omega_R)$ gives that every maximal Cohen-Macaulay $\Lambda$-module has an injective resolution in $\CM(\Lambda)$. This gives a well-defined definition of $\CM$-injective dimension.

\begin{definition}
We say that $\Lambda$ is an $n$-\textit{canonical order} if the canonical module $\omega$ has projective dimension $n$ as a $\Lambda^{\rm op}$-module (equivalently, as a $\Lambda$-module) or equivalently, the regular module $\Lambda$ has CM-injective dimension equal to $n$.
\end{definition}

If $\Lambda$ is an $n$-canonical order, then the injective dimension of $\Lambda$ is equal to $d+n$ \cite[Proposition 1.1(3)]{GotoNishida}. In particular, if $\Lambda$ has finite global dimension, then $\gldim \Lambda = d+n$. So, the CM-injective dimension of $\Lambda$ measures the global dimension, in a sense. When $\Lambda$ is a Gorenstein order, the injective dimension is equal to $d$.
\end{chunk}

\section{Dominant Dimension}

We have seen that Gorenstein orders are in some sense generalizations of self-injective algebras. Self-injective algebras are ubiquitous. To name a few examples, let us start with a finite group $G$ and consider the group algebra $kG$ whose module category is equivalent to the category of $k$-representations of $G$. Other examples include cohomology rings of closed orientable manifolds and Artinian coordinate rings $k[x_1, \ldots, x_n]/(r_1, \ldots, r_n)$ given by regular sequences.

Dominant dimension is a measure of how far away an algebra is from being a self-injective algebra. The pioneers of the theory include Nakayama, Tachikawa, M{\"u}ller and others. The dominant dimension of an algebra $A$ is the minimal number $\ell$ in a minimal injective resolution 
\begin{align*}
    0 \to A \to I^0 \to \ldots \to I^\ell \to \ldots
\end{align*}
such that $I^\ell$ is not projective. If such a number does not exist, we say that the dominant dimension is infinite. For self-injective algebras, the dominant dimension is infinite. The converse of this statement is the famous Nakayama conjecture: It is still an open problem whether infinite dominant dimension implies that the algebra $A$ is self-injective.

In this paper, motivated by the previous section, we will consider the dominant dimension in the category of Cohen-Macaulay modules.

\begin{chunk}\label{setting}
\textbf{Setting.} Throughout this section, we assume that $R$ is a Cohen-Macaulay local ring of Krull dimension $d$ with canonical module and $\Lambda$ is an $R$-order. When we say $\Lambda$ is an $n$-canonical order, we will tacitly assume that $0 \neq n < \infty$. By $\Pi$, we denote a $\Lambda$-module such that $\add \Pi = \add \Lambda \cap \add \omega$. That is, we let $\Pi$ be an additive generator for the subcategory of CM-projective-injective $\Lambda$-modules. We will also assume for convenience that $\Lambda$ is semiperfect and therefore minimal resolutions exist. This can be guaranteed, for instance, by assuming that $R$ is complete.
\end{chunk}

\begin{definition}
We say that $\Lambda$ has \textit{CM-dominant dimension} $\ell$ if $\ell$ is the smallest number in a minimal CM-injective coresolution 
\begin{align*}
    0 \to \Lambda \to I^0 \to \ldots \to I^\ell \to \ldots
\end{align*}
such that $I^\ell$ is not projective (or $\infty$ if such $\ell$ does not exist). We denote the CM-dominant dimension of $\Lambda$ by $\CMdd(\Lambda)$.
\end{definition}
Note that this definition only proves useful in the noncommutative setting. Indeed, in the commutative case, we have $\pd_R \omega_R = 0$ if and only if $R$ is Gorenstein. In this case, the Cohen-Macaulay dominant dimension is $\infty$. Otherwise, that is if $R$ is a non-Gorenstein Cohen-Macaulay local ring, then $\pd_R \omega_R = \infty$ as a consequence of the Auslander-Buchsbaum formula and the CM-dominant dimension is zero. 

It is immediate to see that Gorenstein orders have infinite Cohen-Macaulay dominant dimension. We shall now consider a nontrivial noncommutative example.

\begin{chunk}
\textbf{A nontrivial example.}
Let $k$ be an infinite field and consider $R = k[[x,y,z,u,v]]/I$ where $I$ is the ideal generated by the $2 \times 2$ minors of the generic matrix
\begin{align*}
    \begin{bmatrix}
    x & y & u \\ y & z & v
    \end{bmatrix}.
\end{align*}
This ring is an example of a scroll. Scrolls have nice properties and geometric interpretations for which refer to \cite{eis-har}. In particular, our ring $R$ is an integrally closed Cohen-Macaulay normal domain of Krull dimension 3 and it is a toric isolated singularity. If we consider the $\ZZ$-graded algebra 
\begin{align*}
    S = k[[X_1, X_2, X_3, X_4]] = \bigoplus_{i \in \ZZ} S_i
\end{align*}
with $\deg X_1 = 2, \deg X_2=1$ and $\deg X_3=\deg X_4 = -1$, then we can identify $R$ with $S_0$. 

As an application of the theory of almost split sequences, Auslander and Reiten proved in \cite[Theorem 2.1]{Auslander-Reiten} that $R$ has finite Cohen-Macaulay type: it has only finitely many indecomposable maximal Cohen-Macaulay modules up to isomorphism. Using the notation from \cite[Proposition 16.12]{Y} and the above identification, the indecomposable maximal Cohen-Macaulay modules are $    S_{1}, S_0 \cong R, S_{-1} \cong \omega_R, S_{-2}$ and a rank two module $M$. Let us set notation: We put $X = R \oplus \omega$ and $\Lambda = \End_R(X)$. Then, $\Lambda$ is an $R$-order. We put
\begin{align*}
    D = \Hom_R(-,\omega_R) &: \CM(\Lambda) \to \CM(\Lambda^{\mathrm {op}})\\
    F = \Hom_R(X, - ) &: \module R \to \module \Lambda \\
    G = \Hom_R(-,X) &: \module R \to \module \Lambda^{\mathrm{op}}.
\end{align*}
Note that $F$ restricts to an equivalence $\add X \cong \proj \Lambda$ and $G$ restricts to an equivalence $\add X \cong \proj \Lambda^{\mathrm{op}}$. We are now ready to start. The first thing we do is to understand CM-injective $\Lambda$-modules. To do so, we compute $\omega$.
\begin{align*}
\omega = D \Lambda = D\Hom_R(X,X)=DG(X) = DG(R) \oplus DG(\omega_R).\end{align*}
On the other hand, we have
\begin{align*}
    DG(R) = D\Hom_R(R,X) = DX = \Hom_R(X, \omega_R) = F(\omega_R)
\end{align*}
and
\begin{align*}
    DG(\omega_R) = D\Hom_R(\omega, X) &= D\Hom_R(S_{-1}, S_0 \oplus S_{-1}) \\&= D(S_1 \oplus R) = S_2 \oplus S_{-1} \\&= \Hom_R(S_0 \oplus S_{-1}, S_{-2}) = \Hom_R(X, S_{-2}) = F(S_{-2}).
\end{align*}
Therefore, we conclude that $\omega = F(\omega_R \oplus S_{-2})$. In other words, the only indecomposable CM-injective $\Lambda$-modules are $\Hom_R(X,\omega_R)$ and $\Hom_R(X, S_{-2})$ up to isomorphism as $F$ is fully faithful.

Let us now compute a minimal CM-injective resolution of $\Lambda$. We will start with a projective resolution of $\omega$. We have a short exact sequence
\begin{align*}
    0 \to R \to \omega_R^{\oplus 2} \to S_{-2} \to 0.
\end{align*}
If we apply $F$ to this exact sequence, we get
\begin{align*}
    0 \to FR \to  F\omega_R^{\oplus 2} \to FS_{-2} \to \Ext_R^1(X,R) \cong 0.
\end{align*}
Since $F\omega_R$ is CM-injective, this short exact sequence yields
\begin{align*}
    0 \to FR \to  F\omega_R^{\oplus 3} \to F\omega_R \oplus F{S_{-2}} \cong \omega \to 0
\end{align*}
which is a minimal projective resolution of the $\Lambda$-module $\omega$. Now, applying $D$ to this resolution, we get the following CM-injective resolution of $\Lambda$ in $\CM(\Lambda^{op})$:
\begin{align*}
    0 \to \Lambda \to DF\omega_R^{\oplus 3} \to DFR \to 0.
\end{align*}
We have seen that the only indecomposable CM-injective $\Lambda$-modules are $F\omega_R$ and $FS_{-2}$. Therefore, $FR$ is not a CM-injective $\Lambda$-module and $DFR$ is not a projective $\Lambda^{op}$-module which shows us that the CM-dominant dimension of $\Lambda$ is $1$.
\end{chunk}
\begin{chunk}
\textbf{Cohen-Macaulay Nakayama conjecture.} We have said that the dominant dimension of a finite dimensional algebra measures how far away it is from being a self-injective algebra: If the algebra is self-injective, then the dominant dimension is infinite. We have mentioned that the converse of this statement is the Nakayama conjecture. It is easy to see from our definition of Cohen-Macaulay dominant dimension, Gorenstein orders have infinite CM-dominant dimension. Indeed, the canonical module of a Gorenstein order is projective and projective modules and CM-injective modules coincide. The Cohen-Macaulay Nakayama conjecture states, then, that if an order over a Cohen-Macaulay local ring has infinite dominant dimension, then it must be a Gorenstein order. The following proposotion shows that if one can find a counterexample to the Cohen-Macaulay Nakayama conjecture, then it gives a counterexample to the original Nakayama conjecture.
\begin{proposition}
Let $\Lambda$ be as in Setting \ref{setting} and $x$ be a central nonzerodivisor on $\Lambda$. Then,
\begin{align*}
    \CMdd(\Lambda) = \CMdd(\Lambda/x \Lambda).
\end{align*}
\end{proposition}
\begin{proof}
  It is standard that if the depth of $\Lambda$ is $d$ as an $R$-module, then the depth of $\Lambda/x\Lambda$ as an $R/xR$-module is $d-1$. That is, $\Lambda/x\Lambda$ is an $R/xR$-order and our definitions make sense.
  
  We know that if $P$ is a projective $\Lambda$-module, then $P/xP$ is a projective $\Lambda/x\Lambda$-module. We also have that the canonical module of $R/xR$ is isomorphic to $\omega_R/x\omega_R$. From here, we can conclude that
  \begin{align*}
      \omega_{\Lambda/x\Lambda} = \Hom_{R/xR}(\Lambda/x\Lambda, \omega_{R/xR}) \cong \omega_\Lambda/x\omega_\Lambda.
  \end{align*}
  Now, we can complete the proof by taking a projective resolution of $\omega_\Lambda$ as a $\Lambda^{\rm{op}}$-module and tensoring it with $\Lambda/x\Lambda$ since by doing so still gives us an exact seqeunce and we have already observed that CM-projective-injective $\Lambda$-modules go to CM-projective-injective $\Lambda/x\Lambda$-modules.
\end{proof}
\begin{corollary}
Let $\Lambda$ be as in Setting \ref{setting} and let $\mathbf{x}$ be a regular sequence of length $d$ on $\Lambda$. Then, we have 
\begin{align*}
    \CMdd(\Lambda) = \rm{domdim}(\Lambda/\mathbf{x} \Lambda).
\end{align*}
Hence a counterexample to the Cohen-Macaulay Nakayama conjecture would imply a counterexample to the original Nakayama conjecture.
\end{corollary}

\begin{remark}
It is well-known in the representation theory of finite dimensional algebras that the Nakayama conjecture holds true for those algebras which have finite finitistic dimension. Similar arguments also can be made for orders over Cohen-Macaulay local rings. When $\Lambda$ is an $n$-canonical order over a Cohen-Macaulay local ring of Krull dimension $d$, there is an inequality version of Auslander-Buchsbaum formula which was proved by Stangle in their thesis \cite{Stangle}: for any $\Lambda$-module $X$ with finite projective dimension, we have \[d \leq \pd_\Lambda X + \depth X \leq d + n.\] Hence $n$-canonical orders have finite finitistic dimension. Therefore, if one wants to find a counterexample to the Nakayama conjecture, then they need to look at orders over which the canonical module has infinite projective dimension.
\end{remark}
\end{chunk}

\section{Shifted Orders}

In this section, we keep assuming Setting \ref{setting}. We start with an $R$-order $\Lambda$ of CM-dominant dimension at least $\ell$ so that the minimal CM-injective coresolution of $\Lambda$ is of the form
\begin{align}\label{coresolution}
    0 \to \Lambda \to \Pi^0 \to \Pi^1 \to \ldots \to \Pi^{\ell - 2} \to \Pi^{\ell - 1} \to \ldots
\end{align}
with $\Pi^j$'s CM-projective-injective for $j = 0, 1, \ldots, \ell-1$. We say that $\Lambda$ is a \emph{QF-3 order} if the CM-dominant dimension of $\Lambda$ is at least $1$. The letters QF come from quasi-Frobenius rings and QF-3 rings were introduced by Robert Thrall as a generalization of quasi-Frobenius rings \cite{Thrall}, see also \cite{Tachikawa}.

We denote by $K_\ell$ the cokernel of $\Pi^{\ell -2} \to \Pi^{\ell -1}$. It is clear by definition that the projective dimension of $K_\ell$ is at most $\ell$. 

\begin{lemma}\label{lemmadefiningK}
  The $\Lambda$-module $K_\ell$ is a Cohen-Macaulay $\Lambda$-module.
\end{lemma}

\begin{proof}
  If the CM-injective dimension of $\Lambda$ is finite, then let $X$ be the last nonzero CM-injective module in the coresolution (\ref{coresolution}). In particular, $X$ is a maximal Cohen-Macaulay $R$-module and by applying the depth lemma successively, we see that $K_\ell$ is maximal Cohen-Macaulay as an $R$-module. In case $\Lambda$ has infinite CM-injective dimension, then $K_\ell$ is the $d$th syzygy of a $\Lambda$-module. Once again, the depth lemma does the trick.
\end{proof}

Recall that a $\Lambda$-module $T$ is called an $\ell$-\textit{tilting} module if 
  \begin{enumerate}
    \item the projective dimension of $T$ is at most $\ell$,
    \item there is an exact sequence $0 \to \Lambda \to t_0 \to t_1 \to \ldots \to t_\ell  \to 0$ with $t_0, \ldots, t_\ell \in \add T$,
    \item there are no self-extensions of $T$ in the sense that $\Ext_\Lambda^{>0}(T,T) = 0$.
  \end{enumerate}
If $T$ is a tilting $\Lambda$-module, then $\Lambda$ and $\End_\Lambda(T)$ are derived equivalent. Hence, this is a derived version of Morita theory. The following lemma is about extensions with CM-projective-injective middle terms.

\begin{lemma}\label{ABC-lemma}
  Let $0 \to A \to B \to C \to 0$ be a short exact sequence of maximal Cohen-Macaulay $\Lambda$-modules with $B$ CM-projective-injective. Then, for every $i > 0$, we have
  \begin{align*}
      \Ext_\Lambda^i(A,A) \cong \Ext_\Lambda^i(C,C).
  \end{align*}
\end{lemma}
\begin{proof}
  If we apply $\Hom_\Lambda(-,A)$ to the short exact sequence, we get a long exact sequence
  \begin{align*}
      \ldots \to \Ext_\Lambda^i(B,A) \to \Ext_\Lambda^i(A,A) \to \Ext_\Lambda^{i+1} (C,A) \to \Ext_\Lambda^{i+1}(B,A) \to \ldots
  \end{align*}
  The outer terms vanish as $B$ is a projective $\Lambda$-module and therefore we see that $$\Ext_\Lambda^i(A,A) \cong \Ext_\Lambda^{i+1}(C,A).$$
  
  If we apply $\Hom_\Lambda(C,-)$ to the short exact sequence, we get a long exact sequence
  \begin{align*}
            \ldots \to \Ext_\Lambda^i(C,B) \to \Ext_\Lambda^i(C,C) \to \Ext_\Lambda^{i+1} (C,A) \to \Ext_\Lambda^{i+1}(C,B) \to \ldots
  \end{align*}
  Once again, the outer terms vanish as $B$ is a CM-injective $\Lambda$-module. So, we have isomorphisms \[\Ext_\Lambda^i(C,C) \cong \Ext_\Lambda^{i+1}(C,A). \]
  Combining the two isomorphisms, we get the result.
\end{proof}

\begin{lemma}
The $\Lambda$-module $T_\ell = K_\ell \oplus \Pi$ is an $\ell$-tilting $\Lambda$-module.
\end{lemma}
\begin{proof}
  The first two conditions of a tilting module hold by our definition of $T_\ell$. So, we will show the third condition. Note that we have $\Ext^i_\Lambda(\Pi, - ) = \Ext^i_\Lambda(-,\Pi) = 0$ for $i>0$ since $\Pi$ is CM-projective-injective. Thus, it is enough to show $\Ext_\Lambda^i(K_\ell,K_\ell) = 0$. Note also that by Lemma \ref{ABC-lemma}, it is enough to show $\Ext_\Lambda^i(K_1,K_1) = 0$ since we have a short exact sequence
  \begin{align*}
      0 \to K_j \to \Pi^j \to K_{j+1} \to 0
  \end{align*}
  for every $1 \leq j \leq \ell-1$. 
  
  For $i \geq 2$, we have $\Ext_\Lambda^i(K_1,K_1) = 0$ as $K_1$ has projective dimension $1$. To show that $\Ext_\Lambda^1(K_1,K_1)$ vanishes,  we apply $\Hom_\Lambda(K_1, -)$ to $0 \to \Lambda \to \Pi^0 \to K_1 \to 0$. We get a long exact sequence
  \begin{align*}
    \ldots \to \Ext^1_\Lambda(K_1, \Pi^0) \to \Ext_\Lambda^1(K_1,K_1) \to \Ext_\Lambda^{2}(K_1, \Lambda) \to \ldots .
  \end{align*}
  The outside terms vanish as $\Pi^0$ is CM-injective and projective dimension of $K_1$ is at most $1$.
\end{proof}

\begin{proposition}
The endomorphism ring $\Gamma_\ell = \End_\Lambda(T_\ell)$ is again an $R$-order.
\end{proposition}
\begin{proof}
  We need to show that $\End_\Lambda(T_\ell)$ is maximal Cohen-Macaulay as an $R$-module. Let us start with the direct sum decomposition $\End_\Lambda(T_\ell) \cong \Hom_\Lambda(\Pi, T_\ell) \oplus \Hom_\Lambda(K_\ell, T_\ell)$. We know that $\Pi$ is a projective $\Lambda$-module which tells us that $\Hom_\Lambda(\Pi, T_\ell) \in \add_\Lambda T_\ell$. Hence, $\Hom_\Lambda(\Pi, T_\ell)$ is maximal Cohen-Macaulay since so is $T_\ell$. Therefore, it is enough to show that $\Hom_\Lambda(K_\ell, T_\ell)$ is maximal Cohen-Macaulay. We further decompose this module as
  \begin{align*}
      \Hom_\Lambda(K_\ell, T_\ell) \cong \Hom_\Lambda(K_\ell, K_\ell) \oplus \Hom_\Lambda(K_\ell, \Pi).
  \end{align*}
  Since $\Pi$ is a projective module, we have $\Hom_\Lambda(K_\ell, \Pi)$ is a summand in $\Hom_\Lambda(K_\ell, \Lambda) \cong \Hom_R(K, \omega_R)$ which is a maximal Cohen-Macaulay module. So, it is enough to show that $\Hom_\Lambda(K_\ell, K_\ell)$ is maximal Cohen-Macaulay. We will do this by showing that $\Hom_\Lambda(K_j, K_\ell)$ is maximal Cohen-Macaulay for every $1 \leq j \leq \ell$.
  
  We start with the short exact sequence $     0 \to \Lambda \to \Pi^0 \to K_1 \to 0$ defining $K_1$ and we apply $\Hom_\Lambda( -,K_\ell)$ to it to get an exact sequence
  \begin{align*}
    0 \to \Hom_\Lambda(K_1,K_\ell) \to \Hom_\Lambda(\Pi^{0}, K_\ell) \to \Hom_\Lambda(\Lambda, K_\ell) \to \Ext_\Lambda^1(K_1,K_\ell).
  \end{align*} 
  The rightmost term is zero. Indeed, $K_1$ is the $(\ell-1)^{th}$ syzygy of $K_\ell$ as a $\Lambda$-module. Therefore, $      \Ext_\Lambda^1(K_1,K_\ell) \cong \Ext_\Lambda^\ell(K_\ell, K_\ell) \cong 0
$  by the previous lemma. We have that $\Pi^0$ and $\Lambda$ both projective. So, both $\Hom_\Lambda(\Pi^{0}, K_\ell)$ and $ \Hom_\Lambda(\Lambda, K_\ell)$ are in $\add_\Lambda K_\ell$. Thus, they are both maximal Cohen-Macaulay. By the depth lemma, $\Hom_\Lambda(K_1,K_\ell)$ is also maximal Cohen-Macaulay. Now, assume that the result holds for $j < \ell$. We have a short exact sequence
\begin{align*}
    0 \to K_j \to \Pi^j \to K_{j+1} \to 0
\end{align*}
to which we apply $\Hom_\Lambda(-,K_\ell)$. We get an exact sequence
\begin{align*}
    0 \to \Hom_\Lambda(K_{j+1},K_\ell) \to \Hom_\Lambda(\Pi^{j}, K_\ell) \to \Hom_\Lambda(K_j, K_\ell) \to \Ext_\Lambda^1(K_{j+1},K_\ell).
\end{align*}
  By the same argument as above, the rightmost term is zero and the second term is maximal Cohen-Macaulay. By the induction hypothesis, the third terms is also maximal Cohen-Macaulay. Therefore, by the depth lemma, $\Hom_\Lambda(K_{j+1}, K_\ell)$ is maximal Cohen-Macaulay which finishes the proof.
\end{proof}

\begin{definition}
  We call the module $T_\ell$ the \textit{$\ell^{th}$-shifted module} of $\Lambda$ and $\Gamma_\ell = \End_\Lambda(T_\ell)$ the \textit{$\ell^{th}$-shifted order}.
\end{definition}
We choose the terminology after Matthew Pressland and Julia Sauter \cite[Definition 2.5]{PresslandSauter}.

\begin{lemma}\label{cm-inj-of-shifted}
  The $\ell^{th}$-shifted module $T_\ell$ has CM-injective dimension $n-\ell$.
\end{lemma}
\begin{proof}
  The $\ell^{th}$-shifted module is the direct sum of a CM-projective-injective with the cokernel $K_\ell$ in Lemma \ref{lemmadefiningK}. If $0 \to \Lambda \to \Pi^0 \to \Pi^1 \to \ldots$ is a CM-projective-injective coresolution of $\Lambda$, then $0 \to K_\ell \to \Pi^{\ell} \to  \ldots$ is a CM-injective coresolution of $K$.
\end{proof}

\begin{chunk}
\textbf{Towards the main theorem.} For the purposes of the next three lemmas and the following proposition, we will assume that $\Lambda$ is an $n$-canonical $R$-order, $M$ is a maximal Cohen-Macaulay $\Lambda$-module with CM-injective dimension $m$. Let us denote by $K^{-,b}(\add M)$ the homotopy category of complexes of $\Lambda$-modules with terms in $\add_\Lambda M$, bounded below and with finitely many non-vanishing cohomology groups. We consider a complex $X = (X^i, d^i) \in K^{-,b}(\add M)$ with vanishing nonnegative cohomology and put $L_j = \ker d^j$ for $j \leq 0$. We also let $L_1 = \imof d^0$ and $L_2 = X^1/\imof d^0$. Hence, we have short exact sequences 
  \begin{align}\label{ses}
    0 \to L_j \to X^j \to L_{j+1} \to 0
  \end{align}
  for $j \leq 1$. Assume that $\Ext_\Lambda^i(M,M) = 0$ for any $i > 0$.
  
  \begin{lemma}\label{2-d-lemma}
    With the notation as above, we have that $L_j$ is a maximal Cohen-Macaulay $\Lambda$-module for every $j \leq 2-d$.
  \end{lemma}
  \begin{proof}
    The proof is by applying the depth lemma to the short exact sequences (\ref{ses}).
  \end{proof}

  \begin{lemma}\label{first-isomorphisms}
    Suppose $t \leq 2$. If $i - t \geq m + d - 1$, then $\Ext_{\Lambda}^i(L_t, M) = 0$.
  \end{lemma}
  \begin{proof}
    Applying $\Hom_\Lambda(-,M)$ to the short exact sequences (\ref{ses}) yields long exact sequences
    \begin{align*}
      \ldots \to \Ext^i_\Lambda(X^j, M) \to \Ext_\Lambda^i(L_j, M) \to \Ext_{\Lambda}^{i+1}(L_{j+1},M) \to \Ext_\Lambda^{i+1}(X^j, M) \to \ldots
    \end{align*}
    The outside two terms vanish for $i>0$ as $X^j \in \add M$ and we assumed that $\Ext_\Lambda^i(M,M)$ vanishes for $i > 0$. So, we have isomorphisms
    \begin{align*}
      \Ext_\Lambda^i(L_j, M) \cong \Ext_\Lambda^{i+1}(L_{j+1}, M)
    \end{align*}
    for all $i >0$ and $j \leq 1$. This gives us the vanishings
    \begin{align*}
    \Ext_\Lambda^1(L_{j-m}, M) \cong \ldots   \cong \Ext_\Lambda^m(L_{j-1}, M) \cong \Ext_\Lambda^{m+1}(L_{j}, M) \cong 0
    \end{align*}
    and
    \begin{align*}
       0 \cong  \Ext_\Lambda^{m+1}(L_{j}, M) \cong \Ext_\Lambda^{m+2}(L_{j+1}, M) \cong \ldots \cong \Ext_\Lambda^{m+2-j}(L_{1}, M) 
    \end{align*} 
    for $j \leq 2-d$. In other words, if $i - t = m - 1 - j$, then $\Ext_\Lambda^i(L_t, M) = 0$. This is because we have $L_{j} \in \CM(\Lambda)$ for $j \leq 2-d$ and $M$ has CM-injective dimension $m$. In particular, we have $\Ext_\Lambda^i(L_t, M) = 0$ provided $i - t \geq m+d-1$.
  \end{proof}
  \begin{lemma}\label{second-isomorphisms}
    If $i-t \geq m + d - 1$, then, we have $\Ext_\Lambda^i(L_t, L_{j+1}) \cong \Ext_\Lambda^{i+1}(L_t, L_j)$ for any $j \leq 1$.
  \end{lemma}
    \begin{proof}
      By applying $\Hom_\Lambda(L_t, -)$ to the short exact sequences (\ref{ses}), we get a long exact sequence
    \begin{align*}
              \ldots \to \Ext^i_\Lambda(L_t, X^j) \to \Ext_\Lambda^i(L_t, L_{j+1}) \to \Ext_{\Lambda}^{i+1}(L_t, L_j) \to \Ext_\Lambda^{i+1}(L_t, X^j) \to \ldots.
    \end{align*}
    When $i-t \geq m+d-1$, the outer terms vanish by \ref{first-isomorphisms}. Hence, we have the desired isomorphisms.
    \end{proof}
    \begin{proposition}
    Assume that $\Lambda$ has finite global dimension. If $t \leq 2-d-m$, then $\Ext^1_\Lambda(L_t, L_{j+1}) = 0$ for any $j \leq 1$. In particular, we have
    \begin{align*}
        \Ext^1_\Lambda(L_{2-d-m}, L_{1-d-m}) = 0.
    \end{align*}
    \end{proposition}
    \begin{proof}
      The condition $t \leq 2-d-m$ is equivalent to the condition $1-t \leq m + d - 1$. So, we can apply Lemma \ref{second-isomorphisms} to get
      \begin{align*}
          \Ext_\Lambda^1(L_t, L_{j+1}) \cong \Ext_\Lambda^2(L_t, L_j) \cong \ldots \cong \Ext_\Lambda^n(L_t, L_{j+2-n}) \cong \Ext_\Lambda^{n+1}(L_t, L_{j+1-n}).
      \end{align*}
      Since $t \leq 2-d-m \leq 2-d$, we get that $L_t \in \CM(\Lambda)$. Therefore, its projective dimension is bounded above by $n$. This gives us that $\Ext_\Lambda^{n+1}(L_t, L_{j+1-n})$ vanishes which finishes the proof.
    \end{proof}
    
  \end{chunk}
  
  \begin{corollary}\label{technical-corollary}
  Let $\Lambda$ be an $n$-canonical $R$-order of finite global dimension and $M$ be a $\Lambda$-module such that $\Ext_\Lambda^i(M,M) = 0$ for all $i > 0$. For any $X \in K^{-,b}(\add M)$ with vanishing non-negative cohomology, we have $X = Y \oplus Z$ where $Y$ is acyclic and $Z^i = 0$ for all $i < 1-m-d$ where $m$ is the CM-injective dimension of $M$.
  \end{corollary}
  \begin{proof}
    With the notation as above, we have short exact sequences
    \begin{align*}
        0 \to L_j \to X^j \to L_{j+1} \to 0
    \end{align*}
    and we now know that $\Ext^1_\Lambda(L_{2-d-m}, L_{1-d-m}) = 0$ so that the sequence
    \begin{align*}
        0 \to L_{1-d-m} \to X^{1-d-m} \to L_{2-d-m} \to 0
    \end{align*}
    splits and we have a decomposition $X^{1-d-m} \cong L_{1-d-m} \oplus L_{2-d-m} $. In particular, we have that $L_{1-d-m}, L_{2-d-m} \in \CM(\Lambda)$.
    Therefore, choosing $Y$ and $Z$ as follows give us the desired direct sum decomposition.
    \begin{align*}
      \xymatrix{
	Y = \ldots \ar[r]& X^{-1-d-m} \ar[r]& X^{-d-m} \ar[r]& L_{1-d-m} \ar[r]& 0 \ar[r]& 0 \ar[r]& \ldots \\
	Z = \ldots \ar[r]& 0 \ar[r] & 0 \ar[r] & L_{2-d-m} \ar[r] & X^{2-d-m} \ar[r] & X^{3-d-m} \ar[r] & \ldots
      }
    \end{align*}
  \end{proof}
  
  \begin{theorem}\label{key-to-main-theorem}
  Suppose that $\Lambda$ is an $R$-order of finite global dimension and $M \in \CM(\Lambda)$ is an $\ell$-tilting module of CM-injective dimension $m$. Then, we have
  \begin{align*}
      \gldim \End_\Lambda(M) \leq \ell + m + d.
  \end{align*}
  \end{theorem}
  \begin{proof}
    Let us start with putting $\Gamma = \End_\Lambda(M)$. Since $\Lambda$ is of finite global dimension and $\Gamma$ is derived equivalent to $\Lambda$, we know that $\Gamma$ also has finite global dimension. We will show that any $X \in \module \Gamma$ has projective dimension at most $\ell + m + d$.
    
    Let $P$ be a minimal projective resolution of $X$. Since $P$ is minimal, the width of $P$ is $\pd X + 1$. The equivalence between $D^b(\Lambda)$ and $D^b(\Gamma)$ restricts to an equivalence between $K^{-,b}(\add M)$ and $K^{-,b} (\proj \Gamma)$. The image of $P \in K^{-,b}(\proj \Gamma)$ under this equivalence is $M \otimes_{\Gamma} P$. Since $P$ is minimal, it has no nonzero acyclic summands and the same holds for $M \otimes_{\Gamma} P$. 
    
    Since $M$ is a tilting $\Lambda$-module of projective dimension at most $\ell$, it has projective dimension at most $\ell$ as a $\Gamma$-module. Hence,
  \begin{align*}
  H^i(M \otimes_{\Gamma} P) = \Tor_{-i}^{\Gamma}(M,X) = 0
  \end{align*}
  for $i < - \ell$. Thus, $M \otimes_{\Gamma} P[-\ell -1 ]$ has no non-negative cohomology and we can apply Corollary \ref{technical-corollary}. This gives us a direct sum decomposition $M \otimes_{\Gamma} P[-\ell -1 ] = Y \oplus Z$ where $Y$ is acylic and $Z^i = 0$ for $i < 1 - m- d$. But since $M \otimes_{\Gamma} P$ does not have any nonzero acylic summands, we must have $Y = 0$. Therefore, we actually have $M \otimes_{\Gamma} P[-1-\ell] \cong  Z$. So, we can conclude that $Z^i$ is only allowed to be nonzero in the interval $1-d-m \leq i \leq \ell + 1$. Consequently, we have an upper bound on the width of $M \otimes_\Gamma P$. Namely, $\ell + m + d + 1$. This means that the width of $P$ is bounded above by $\ell + m + d + 1$ and thus the projective dimension of $X$ is bounded above by $\ell + m + d$.
  \end{proof}
  
  \begin{corollary}\label{maintheorem}
  Suppose that $\Lambda$ is an $n$-canonical QF-3 order of finite global dimension. Then, the global dimension of the $\ell^{\rm{th}}$-shifted order $\Gamma_\ell$ is bounded above by the global dimension of $\Lambda$.
  \end{corollary}
  \begin{proof}
  The $\ell^{\rm{th}}$-shifted module $T_\ell$ is an $\ell$-tilting module in $\CM(\Lambda)$ which has CM-injective dimension $n- \ell$ by Lemma \ref{cm-inj-of-shifted}. Thus by Theorem \ref{key-to-main-theorem}, we get that 
  \begin{align*}
      \gldim \Gamma_\ell \leq n + d = \gldim \Lambda.
  \end{align*}
  \end{proof}
  
\bibliographystyle{alpha}
\bibliography{arxiv_version}

\end{document}